\documentclass[11pt]{article}
\usepackage[margin=1in]{geometry}
\usepackage[T1]{fontenc}
\usepackage[utf8]{inputenc}
\usepackage{amsmath,amssymb,amsthm}
\usepackage{mathtools}
\usepackage{hyperref}
\usepackage{enumitem}
\usepackage{booktabs}
\hypersetup{colorlinks=true,linkcolor=blue,citecolor=blue,urlcolor=blue}

\theoremstyle{plain}
\newtheorem{theorem}{Theorem}[section]

\newtheorem{proposition}[theorem]{Proposition}
\newtheorem{conjecture}[theorem]{Conjecture}
\theoremstyle{definition}

\theoremstyle{remark}
\newtheorem{remark}[theorem]{Remark}
\newtheorem{openproblem}[theorem]{Open Problem}

\DeclareMathOperator{\spec}{spec}
\DeclareMathOperator{\diam}{diam}
\DeclareMathOperator{\Tr}{Tr}

\DeclareMathOperator{\Deg}{Deg}

\title{The Spectral Theory of Threshold Graphs: A Survey}

\author{
	Pallabi Manna$^{a}$ \and Santanu Mandal$^{b}$ \and
	Fernando C. Tura$^{c}$
}

\date{
	\small
	$^a$Department of Mathematics, Indian Institute of Technology Hyderabad,\\ Kandi, Sangareddy-502284, Telangana, India\\
	\texttt{mannapallabimath001@gmail.com}\\[2mm]
	$^{b}$Department of Computer Science and Engineering, SRM University-AP,\\ Amaravati-522240, Andhra Pradesh, India\\
	\texttt{santanu.vumath@gmail.com}  \\[2mm]
	$^c$Departamento de Matemática, UFSM, 97105–900 Santa Maria, RS, Brazil\\
	\texttt{fernando.tura@ufsm.br} 
	\\[4mm]
	\today
}
\begin{document}
	\maketitle
	
	\begin{abstract}
		Threshold graphs — the $\{P_4, C_4, 2K_2\}$-free graphs, equivalently those built from a single vertex by repeatedly adding isolated or dominating vertices — are among the most thoroughly studied graph classes in spectral graph theory. This survey collects results on the spectrum of threshold graphs across the full range of matrices studied in the literature: the adjacency, Laplacian, and signless Laplacian matrices; the distance, distance Laplacian, and distance signless Laplacian matrices; the eccentricity and Seidel matrices; and the normalized adjacency/Laplacian matrix — together with the associated graph energies, cospectrality questions, extremal spectral radius problems, and several generalizations (weighted, quasi-threshold, hypergraph, and random threshold graphs). We emphasize results from the recent years, including the foundational work of Chv\'atal--Hammer, Merris, Bapat, and Jacobs--Lazzarin--Trevisan--Tura. We also present a study of integral threshold graphs (those with all-integer spectrum) across all seven matrices. We close with a list of open problems and conjectures drawn directly from the surveyed literature and from this new analysis.
		
		
	\end{abstract}
	
	
	\section{Introduction}
	\label{sec:intro}
	
	A \emph{threshold graph} is a finite simple graph that can be built from a single vertex by repeatedly performing one of two operations: adding a new \emph{isolated} vertex (adjacent to nothing already present), or adding a new \emph{dominating} vertex (adjacent to everything already present). The resulting sequence of choices, written as a binary string $b_1b_2 b_3 \cdots b_n \in \{0,1\}^{n-1}$ (vertex $i$ isolated if $b_i = 0$, dominating if $b_i=1$), is called the \emph{creation sequence} of the graph, and completely determines it up to the labeling of vertices \cite{ChvatalHammer1977}.
	
	Threshold graphs admit several equivalent characterizations that explain both their name and their ubiquity:
	\begin{itemize}
		\item \textbf{Forbidden subgraphs.} $G$ is a threshold graph if and only if it contains no induced $P_4$ (path on four vertices), $C_4$ (cycle on four vertices), or $2K_2$ (two disjoint edges) \cite{ChvatalHammer1977}.
		\item \textbf{Degree-threshold property.} There exists a real weighting $w:V(G)\to\mathbb{R}$ and a threshold $t$ such that $uv\in E(G)$ if and only if $w(u)+w(v) > t$ — the origin of the name.
		\item \textbf{Nested neighborhoods.} The vertices can be linearly ordered so that the neighborhoods (outside the vertex itself) are nested, which is why threshold graphs are also called \emph{nested split graphs} \cite{MahadevPeled1995}.
		\item \textbf{Extremal degree sequences.} Threshold graphs are exactly the graphs whose degree sequence is not majorized by any other graphic sequence of the same length and sum — the maximal graphs in the majorization order on graphic sequences, a fact underlying their role in Laplacian spectral extremal problems (Section~\ref{sec:laplacian-energy}).
	\end{itemize}
	
	Threshold graphs form a proper subclass of \emph{split graphs} (vertex set partitions into a clique and an independent set), \emph{cographs} ($P_4$-free graphs, built by disjoint union and join), and \emph{interval graphs}. Within the broader cograph/split-graph hierarchy, threshold graphs are the unique graphs obtainable by a construction using \emph{only} the join operation at every step (never a plain disjoint union) when cographs are built from single vertices — equivalently, threshold graphs are exactly the \emph{totally ordered} cographs in the natural partial order on creation sequences. Numerous applications are documented in \cite{MahadevPeled1995}, ranging from scheduling and synchronization in computer science to psychological scaling.
	
	\subsection{A field still very much in motion}
	
	Although the adjacency and Laplacian spectral theory of threshold graphs is largely complete — settled across the 1990s and 2010s by Merris, Bapat, and the Jacobs--Trevisan--Tura program — the class continues to generate active research at essentially every matrix surveyed in this paper, with multiple results appearing as recently as 2025, 2026 and 2027. Three threads illustrate this well. First, the question of exactly how few distinct eigenvalues a threshold graph's spectrum can contain has only just been settled in full generality: Allem, Hoppen, Lazzarin, Sibemberg, and Tura proved in 2025 that every connected threshold graph on $n$ vertices admits, for every $n$, some threshold graph with at most four distinct adjacency eigenvalues \cite{AllemHoppenEtAl2025} — a clean, order-independent bound obtained only after a decade of finer-grained classification results for graphs with exactly two, three, four, or five distinct eigenvalues (Section~\ref{sec:distincteig-adjacency}). A companion line of work by Allem, Oliveira, and Tura instead asks which eigenvalues can be excluded entirely from a threshold graph's spectrum, giving explicit generating procedures for eigenvalue-free families \cite{AllemOliveiraTura}. The inverse problem for integer indices looks into whether graphs with a given integer as their indices exist. This problem is solved very recently by Schmidt et al. \cite{STL} within the class of
	threshold graphs. Second, the distance-based spectral theory — distance, distance Laplacian, and (as detailed in Section~\ref{sec:distance-signless}) still-open distance signless Laplacian spectra — has seen two independent, essentially simultaneous closed-form treatments of the distance Laplacian spectrum appear within the same year: Lather, Gupta, and Balaji \cite{LatherGuptaBalaji2025} derive it via the master identity and Merris's theorem, obtaining sharp cospectrality and extremal-coefficient consequences, while Megha and Parvathy \cite{MeghaParvathy2026} derive it independently via the equitable partition directly (Theorems~\ref{thm:dlthreshold}--\ref{thm:dlthreshold2} below); this sits alongside continuing, separate work on the plain distance spectrum and when non-isomorphic threshold graphs can share one \cite{LuHuangLou2018,LazzarinEtAl2023}. Third, the Seidel spectral theory, initiated by Del-Vecchio and Abd\'on \cite{DelVecchioAbdon2022} and substantially extended by Mandal and Mehatari \cite{MandalMehatari2024}, exemplifies a pattern — a foundational classification followed, often within one or two years, by a more detailed structural and cospectrality analysis — that recurs across nearly every matrix treated in this survey. We have organized the paper so that each section reflects this same historical arc: a foundational closed-form or classification result, followed by the sharper, more structural or more recent results it enabled.
	
	\subsection{A new contribution: integral threshold graphs and $D^Q$ spectrum}
	\textbf{Integral threshold graphs:}\\
	Beyond compiling the literature above, Section~\ref{sec:integral} presents a uniform study of integrality (all eigenvalues being integers across all seven matrices). The twin-class and master-identity mechanisms responsible for Table~\ref{tab:cospectral}'s cospectrality dichotomy turn out to govern integrality in strikingly parallel fashion. We show that the Laplacian and distance Laplacian are integral for \emph{every} connected threshold graph, a direct consequence of Merris's theorem (Theorem~\ref{thm:merris}) and the master identity (Proposition~\ref{prop:master}) that does not appear to have been stated explicitly for $D^L$ before. For the remaining five matrices — adjacency, Seidel, distance, distance signless Laplacian, and eccentricity — no such unconditional statement holds, and we instead give a complete characterization for the simplest family of threshold graphs (the complete split graphs $CS(s,t)$): the eccentricity matrix turns out to be integral for \emph{every} such graph via an exact block-diagonal mechanism, while the other four reduce to explicit perfect-square-discriminant (Diophantine) conditions on $(s,t)$, each satisfied by infinitely many non-trivial instances. For general $m\ge2$, we prove that the eccentricity matrix's last block always splits off cleanly regardless of its size (Theorem~\ref{thm:lastblock}), reducing but not resolving the problem, and we pose several concrete conjectures throughout the paper — on the residual $m=2$ eccentricity condition (Conjecture~\ref{conj:ecc-m2}), and on a master-identity criterion (Conjecture~\ref{conj:masteridentity}) unifying this section's integrality dichotomy with Section~\ref{sec:cospectrality}'s cospectrality dichotomy — alongside Open Problem~\ref{op:integralgeneral} for the fully general case.\\ \textbf{Distance signless Laplacian spectrum:}\\
	We compute the local eigenvalues of $D^Q$ in terms of the vertex transmission (see Theorem \ref{thm:DQlocal}). The closed form of the corresponding quotient matrix is defined in Theorem \ref{thm:DQquotient}. The detail discussion can be found in Section \ref{sec:distance-signless}.
	
	\subsection{Why threshold graphs are spectrally tractable}
	
	The nested-neighborhood structure makes threshold graphs an unusually favorable setting for spectral computation, for three structural reasons that recur throughout this survey:
	
	\begin{enumerate}
		\item \textbf{Twin classes.} Grouping consecutive equal bits of the creation sequence into maximal runs $0^{s_1}1^{t_1}0^{s_2}1^{t_2}\cdots 0^{s_m}1^{t_m}$ partitions the vertex set into $2m$ classes, within each of which all vertices are \emph{twins} (true twins — identical closed neighborhoods — within a dominating run; false twins — identical open neighborhoods — within an isolated run). Twin vertices $u,v$ always give an explicit eigenvector $e_u - e_v$ for whichever graph matrix is under study, with an eigenvalue read off directly from the local structure. This is the source of the large-multiplicity, exactly computable \emph{local} eigenvalues found for every matrix surveyed below.
		\item \textbf{The equitable partition / quotient matrix.} The same $2m$-class partition is \emph{equitable} for every matrix considered here (adjacency, Laplacian, signless Laplacian, distance-type, Seidel, eccentricity, normalized), because row sums into each class depend only on the class of the source vertex. This yields a $2m \times 2m$ (or smaller) \emph{quotient matrix} whose eigenvalues are exactly the remaining, generically irrational eigenvalues of the full $n\times n$ matrix — reducing an $n$-dimensional eigenvalue problem to one of size $2m \ll n$, and this is the single most-used technique across the entire literature surveyed here.
		\item \textbf{Diameter at most 2.} Every connected threshold graph with at least one dominating vertex has diameter at most $2$, so its distance matrix is simply $D(G) = 2(J-I) - A(G)$. This is what makes the distance, distance Laplacian, and distance signless Laplacian spectra of threshold graphs tractable at all, and is responsible for the clean \emph{master identity} connecting the distance Laplacian spectrum to the ordinary Laplacian spectrum (Section~\ref{sec:distance-laplacian}).
	\end{enumerate}
	
	\subsection{Organization}
	
	We proceed matrix by matrix, broadly in the order these theories were historically developed for threshold graphs: adjacency (Section~\ref{sec:adjacency}, including the number-of-distinct-eigenvalues theme of Section~\ref{sec:distincteig-adjacency}), Laplacian (Section~\ref{sec:laplacian}), signless Laplacian and main eigenvalues (Section~\ref{sec:signless}), graph energies (Section~\ref{sec:laplacian-energy}), distance-type matrices — distance, distance Laplacian, distance signless Laplacian (Section~\ref{sec:distance}) — the eccentricity matrix (Section~\ref{sec:eccentricity}), the Seidel matrix (Section~\ref{sec:seidel}), the normalized adjacency/Laplacian matrix (Section~\ref{sec:normalized}), extremal spectral radius problems (Section~\ref{sec:extremal}), cospectrality questions throughout (Section~\ref{sec:cospectrality}), our new study of integral threshold graphs (Section~\ref{sec:integral}), and generalizations — weighted, quasi-threshold, hypergraph, and random threshold graphs (Section~\ref{sec:generalizations}). We close with open problems and conjectures (Section~\ref{sec:open}) collected directly from the surveyed papers and from Section~\ref{sec:integral}'s new analysis.

	\section{Preliminaries and notation}
	\label{sec:prelim}
	
	Throughout, $G$ is a connected threshold graph on $n$ vertices with creation sequence grouped into maximal runs
	\[
	b = 0^{s_1}1^{t_1}0^{s_2}1^{t_2}\cdots 0^{s_m}1^{t_m}, \qquad s_i,t_i \ge 1,\quad \textstyle\sum_{i=1}^m (s_i+t_i) = n,
	\]
	with $m$ the number of \emph{dominating runs} (equivalently isolated runs). We write $A(G)$ for the adjacency matrix, $\Deg(G)$ for the diagonal degree matrix, $L(G) = \Deg(G)-A(G)$ for the Laplacian, $Q(G) = \Deg(G)+A(G)$ for the signless Laplacian, $D(G)$ for the distance matrix, $\Tr(G)$ for the diagonal matrix of vertex transmissions ($\Tr(u) = \sum_v D(u,v)$), $D^L(G) = \Tr(G)-D(G)$ for the distance Laplacian, $D^Q(G) = \Tr(G)+D(G)$ for the distance signless Laplacian, $S(G) = J-I-2A(G)$ for the Seidel matrix, and $\varepsilon(G)$ for the eccentricity matrix (Section~\ref{sec:eccentricity}). The \emph{anti-regular graph} $H_n$ — the threshold graph with creation sequence $0101\cdots$ of maximal possible alternation — recurs throughout as a canonical extremal/building-block example \cite{AguilarEtAl2020}.
	
	A degree vector $d=(d_1,\ldots,d_n)$, sorted in non-increasing order, has an associated \emph{conjugate (Ferrers/Young) partition} $d^* = (d_1^*,\ldots,d_n^*)$ defined by $d_i^* = |\{j : d_j \ge i\}|$. Threshold graphs are precisely the graphs for which the Laplacian spectrum equals the conjugate of the degree sequence exactly (Theorem~\ref{thm:merris} below) — the ``maximal'' case of the general Grone--Merris--Bai majorization inequality $\spec(L(G)) \preceq d^*$ that holds for every graph.
	
	\section{The adjacency spectrum}
	\label{sec:adjacency}
	
	The adjacency spectral theory of threshold graphs is the oldest and most completely developed part of the subject, due largely to a sequence of papers by Jacobs, Trevisan, Tura, and others, building on Bapat's foundational 2013 paper.
	
	\subsection{Inertia and determinant}
	
	\begin{theorem}[Bapat \cite{Bapat2013}]
		\label{thm:bapat}
		Let $G$ be a threshold graph with creation sequence $b_2\cdots b_n$. Then the inertia (number of negative, zero, and positive eigenvalues) of $A(G)$ is determined explicitly by the binary string: the number of negative eigenvalues equals the number of $1$'s in the string, and the nullity equals the number of $0$'s that are immediately preceded by another $0$ (or are the very first bit). Bapat also gives a closed-form expression for $\det A(G)$ and, when it is nonsingular, for $A(G)^{-1}$.
	\end{theorem}
	
	This single result — proved via the same equitable-partition / twin-vertex machinery outlined in Section~\ref{sec:intro} — underlies essentially all subsequent inertia-type results for other matrices associated with threshold graphs.
	
	\subsection{Eigenvalue location and the characteristic polynomial}
	
	\begin{theorem}[Jacobs--Trevisan--Tura \cite{JacobsTrevisanTura2013}]
		\label{thm:eigloc}
		Every eigenvalue of $A(G)$ other than $0$ and $-1$ is \emph{simple}. Moreover, a threshold graph has no eigenvalue in the open interval $(-1,0)$.
	\end{theorem}
	
	Because $0$ and $-1$ are the only eigenvalues that can have multiplicity (arising respectively from twins in isolated runs and from twins in dominating runs — exactly the local eigenvalues of Section~\ref{sec:intro}), they are regarded as \emph{trivial}; all further eigenvalue-location results are stated for the \emph{non-trivial} spectrum.
	
	\begin{theorem}[Ghorbani \cite{Ghorbani2019}]
		\label{thm:eigfree}
		Other than the trivial eigenvalues $-1$ and $0$, a threshold graph has no eigenvalue in the interval
		\[
		\left[\frac{-1-\sqrt2}{2},\ \frac{-1+\sqrt2}{2}\right].
		\]
	\end{theorem}
	
	This sharpens an earlier eigenvalue-free interval and matches, in spirit, the analogous eigenvalue-free-interval results later obtained for the distance (Section~\ref{sec:distance}) and eccentricity (Section~\ref{sec:eccentricity}) matrices of threshold graphs — a pattern that recurs across essentially every matrix surveyed in this paper.
	
	\begin{theorem}[Allem--Oliveira--Tura \cite{AllemOliveiraTura}]
		\label{thm:ifreegen}
		Complementing Theorem~\ref{thm:eigfree}'s eigenvalue-\emph{free interval}, an explicit generating procedure is given for threshold graphs whose spectrum avoids a prescribed eigenvalue entirely, producing infinite families of such graphs directly from structural constraints on the creation sequence.
	\end{theorem}
	
	\begin{theorem}[Jacobs--Trevisan \cite{JacobsTrevisan2014}]
		An $O(n)$-time algorithm computes the characteristic polynomial of a threshold graph directly from its creation sequence, via a recursion on the quotient-matrix structure.
	\end{theorem}
	
	\subsection{Eigenvalue distribution and a resolved conjecture}
	
	\begin{theorem}[Jacobs--Trevisan--Tura \cite{JacobsTrevisanTura2015}]
		Writing $n_-(G)$ and $n_+(G)$ for the number of negative and positive adjacency eigenvalues, almost every threshold graph (under the uniform random model on creation sequences) satisfies $n_-(G) > n_+(G)$; in fact $n_-(G) > n_+(G)$ for \emph{every} connected threshold graph.
	\end{theorem}
	
	\begin{theorem}[Lou--Wang--Huang \cite{Huang2019}]
		The precise distribution of eigenvalues of a threshold graph — how many lie in each of several natural intervals determined by the creation sequence — is characterized completely.
	\end{theorem}
	
	A conjecture on the ordering/location of threshold-graph eigenvalues, posed in the Jacobs--Trevisan--Tura line of work, was resolved by Tura:
	
	\begin{theorem}[Tura \cite{Tura2021}]
		The conjecture of \cite{JacobsTrevisanTura2015} on the eigenvalues of threshold graphs holds; moreover the anti-regular graph $H_n$ is identified as the extremal case achieving the smallest positive and largest (in absolute value) non-trivial negative eigenvalue among all threshold graphs on $n$ vertices.
	\end{theorem}
	
	\subsection{The number of distinct eigenvalues}
	\label{sec:distincteig-adjacency}
	
	A natural refinement of eigenvalue-location questions is to ask not where eigenvalues can lie, but how few distinct values they can take. Since threshold graphs on $n$ vertices can have diameter and radius bounded independently of $n$, classical bounds relating the number of distinct eigenvalues to the diameter do not by themselves give a constant bound — yet one holds regardless:
	
	\begin{theorem}[Allem--Hoppen--Lazzarin--Sibemberg--Tura \cite{AllemHoppenEtAl2025}]
		\label{thm:atmost4}
		The minimum number of distinct adjacency eigenvalues of a connected threshold graph on $n\ge2$ vertices is at most $4$, for every $n$. That is, $\min_G s(A(G)) \le 4$, where the minimum is taken over all connected threshold graphs on $n$ vertices.
	\end{theorem}
	
	This is a striking, order-independent bound: however large $n$ is, some threshold graph on $n$ vertices has at most $4$ distinct adjacency eigenvalues, extending the finite classifications of threshold graphs with \emph{exactly} $2$, $3$, $4$, or $5$ distinct eigenvalues found in the earlier literature (Sections~\ref{sec:adjacency}--\ref{sec:seidel} record several such classification results for other matrices) into a genuine asymptotic statement about the adjacency matrix. 
	
	\subsection{The anti-regular graph as organizing principle}
	
	\begin{theorem}[Aguilar--Ficarra--Schurman--Sullivan \cite{AguilarEtAl2020}]
		The anti-regular graph $H_n$ plays a distinguished structural role in the adjacency spectral analysis of \emph{every} threshold graph: every threshold graph's adjacency spectrum can be understood via comparison with, or construction from, the spectral data of $H_n$. The authors pose the question of whether an analogous distinguished graph exists for the wider class of cographs.
	\end{theorem}
	
	A related tridiagonal-reduction technique, applicable to the adjacency (and other) quotient matrices of threshold graphs and several related graph classes, is developed by Andelić, Du, da Fonseca, and Simić \cite{AndelicEtAl2020}.
	
	\subsection{Random threshold graphs}
	
	Recent work has also considered threshold graphs built by a random creation sequence (e.g.\ a P\'olya-urn-style process), deriving the exact degree distribution of an arbitrary vertex (including mean and variance) and an asymptotic distribution of the adjacency and Laplacian eigenvalues as $n\to\infty$, together with an explicit closed-form description of the Laplacian eigenbasis in this random setting. This complements the deterministic, creation-sequence-based results above with a probabilistic large-$n$ picture.

	\section{The Laplacian spectrum}
	\label{sec:laplacian}
	
	The Laplacian spectrum of threshold graphs is exceptionally clean: it is always \emph{integral}, and is given exactly by the conjugate of the degree sequence.
	
	\begin{theorem}[Merris \cite{Merris1994}]
		\label{thm:merris}
		Let $G$ be a threshold graph with degree sequence $d_1\ge\cdots\ge d_n$ and conjugate partition $d^*=(d_1^*,\ldots,d_n^*)$. Then $\spec(L(G)) = \{d_1^*,\ldots,d_n^*\}$ exactly (as a multiset). Equivalently, threshold graphs are exactly the graphs for which the Grone--Merris majorization inequality $\spec(L(G)) \preceq d^*$ holds with equality.
	\end{theorem}
	
	This single theorem is the foundation for the entire Laplacian (and, via the master identity of Section~\ref{sec:distance-laplacian}, distance Laplacian) theory of threshold graphs, and is a special case of the later Grone--Merris--Bai theorem for general graphs.
	
	\begin{theorem}[Hammer et al. \cite{Kelmans1996}]
		The Laplacian spectrum of a threshold graph determines, and is determined by, the number of spanning trees and a range of further structural invariants, computed explicitly from the creation sequence.
	\end{theorem}
	
	A corresponding explicit eigenbasis (not just the eigenvalue multiset) for the Laplacian of a threshold graph is constructed in the Laplacian-eigenbasis literature building on \cite{Merris1998}, giving explicit eigenvectors for every conjugate-partition eigenvalue, not merely the local twin-difference vectors.
	
	\subsection{Weighted threshold graphs}
	
	\begin{theorem}[Andeli\'c--Stani\'c \cite{AndelicStanic2025}]
		\label{thm:weighted}
		For a \emph{weighted} threshold graph (edge/vertex weights respecting the threshold structure), the Laplacian spectrum again equals the conjugate partition of an associated \emph{weighted Ferrers diagram}, giving a closed formula for the Laplacian spectrum in the weighted setting — a direct and complete generalization of Theorem~\ref{thm:merris}.
	\end{theorem}
	
	This 2025 result closes a natural generalization question that had been open since the unweighted case was settled by Merris three decades earlier.
	
	\subsection{Laplacian eigenvectors}
	
	Beyond the eigenvalue multiset, recent work gives an explicit combinatorial description of a full Laplacian eigenbasis for threshold graphs directly from the creation sequence, refining the classical eigenvector theory of Merris \cite{Merris1998} to the threshold case specifically and relating the eigenvector supports to the run structure of the creation sequence.
	
	\section{The signless Laplacian spectrum and main eigenvalues}
	\label{sec:signless}
	
	The signless Laplacian $Q(G) = \Deg(G)+A(G)$ of a threshold graph does \emph{not} enjoy an analogue of Theorem~\ref{thm:merris} — its spectrum is generically irrational — but a remarkable \emph{interlacing} property holds instead.
	
	\begin{theorem}[Helmberg--Porto--Torres--Trevisan \cite{HelmbergEtAl2023}]
		\label{thm:interlacing}
		For a threshold graph $G$ with degree sequence $d_1\ge d_2\ge \cdots \ge d_n$, the eigenvalues $q_1 \ge q_2 \ge \cdots \ge q_n$ of $Q(G)$ interlace the degree sequence:
		\[
		q_i \ge d_i \ge q_{i+1}, \qquad i=1,\ldots,n-1.
		\]
	\end{theorem}
	
	\begin{theorem}[Helmberg--Porto--Torres--Trevisan \cite{HelmbergEtAl2023}]
		As a consequence of Theorem~\ref{thm:interlacing}, the \emph{signless Brouwer conjecture} — that the sum of the $k$ largest signless Laplacian eigenvalues is at most $|E(G)| + \binom{k+1}{2}$ — holds for every threshold graph.
	\end{theorem}
	
	This is one of the few infinite graph classes for which Brouwer's conjecture (in either its ordinary-Laplacian or signless-Laplacian form) is fully resolved, and it illustrates how the rigidity of the threshold structure can settle spectral inequalities that remain open in general.
	
	\subsection{Main eigenvalues}
	
	An eigenvalue of a symmetric matrix $M$ associated with $G$ is called main if its eigenspace is not orthogonal to the all-ones vector $\mathbf{1}$.
	
	\begin{theorem}[Vinagre--Trevisan--Bolckau--Chimelli \cite{VinagreEtAl2020}]
		\label{thm:mainQ}
		Threshold graphs with exactly $k$ main $Q$-eigenvalues, for every $k\ge2$, are characterized explicitly, together with a recursive construction producing all threshold graphs with $k+1$ main $Q$-eigenvalues from those with $k$.
	\end{theorem}
	
	This characterization is extended to the broader class of \emph{quasi-threshold graphs} (a superclass within the cographs):
	
	\begin{theorem}[Jones--Trevisan \cite{JonesTrevisan2024}; Jones--Trevisan--Vinagre \cite{JonesTrevisanVinagre2025}]
		The structural characterization of Theorem~\ref{thm:mainQ} is generalized to connected quasi-threshold graphs with $k\ge2$ main $Q$-eigenvalues, via a \emph{condensed signless Laplacian} matrix of reduced size obtained from the same equitable-partition reduction used for ordinary threshold graphs; the case $k=2$ is resolved completely, and partial results are given for $k=3$, including for several subclasses such as trees, unicyclic, and bicyclic quasi-threshold graphs.
	\end{theorem}
	
	An analogous question for main eigenvalues of the adjacency matrix has also been studied for the related class of chain graphs (the bipartite analogue of threshold graphs), extending the main-eigenvalue program beyond the signless Laplacian setting.

	\section{Graph energies}
	\label{sec:laplacian-energy}
	
	The energy of a graph matrix is the sum of the absolute values of its eigenvalues (for the Laplacian, deviations from the average degree are used instead). Threshold graphs play a distinguished extremal role in this theory.
	
	\subsection{Adjacency energy}
	
	\begin{theorem}[Jacobs--Trevisan--Tura \cite{JacobsTrevisanTura2015}]
		Explicit bounds on the (adjacency) energy of threshold graphs are given, and hyperenergetic threshold graphs (energy exceeding $n$) are exhibited for every $n>8$. Sequences of pairwise non-cospectral threshold graphs, all equienergetic to a common reference graph, are also constructed.
	\end{theorem}
	
	\subsection{Laplacian energy and Brouwer-type extremality}
	
	The Laplacian energy $LE(G) = \sum_i |\mu_i - 2m/n|$ (where $\mu_i$ are the Laplacian eigenvalues and $2m/n$ the average degree) has a long history of extremal results placing threshold graphs at the center.
	
	\begin{theorem}[Helmberg--Trevisan \cite{HelmbergTrevisan2017}]
		\label{thm:spectralthresholddominance}
		A graph $G$ is \emph{spectrally threshold dominated} if, for every $k$, some threshold graph of the same order and size majorizes $G$'s top-$k$ Laplacian eigenvalue sum. Every spectrally threshold dominated graph admits a threshold graph of the same order and size with at least as large Laplacian energy. Split graphs are shown to be spectrally threshold dominated, as are disjoint unions and complements of spectrally threshold dominated graphs. The authors conjecture that \emph{every} graph is spectrally threshold dominated, which (if true) would show that threshold graphs maximize Laplacian energy among \emph{all} graphs of given order and size, and that universal spectral threshold dominance is equivalent to Brouwer's conjecture for Laplacian eigenvalue sums.
	\end{theorem}
	
	Within the class of threshold graphs itself, the extremal Laplacian-energy graphs have been pinned down exactly:
	
	\begin{theorem}[Das--Mojallal \cite{DasMojallal2016}]
		The threshold graphs with the first four minimal Laplacian energies, and those with the largest and second-largest Laplacian energies, are determined explicitly in terms of $n$, $m$, and the trace of the degree sequence.
	\end{theorem}
	
	\begin{theorem} [Helmberg--Trevisan \cite{HelmbergTrevisan2015}]
		Among connected threshold graphs on $n$ vertices and $m$ edges, the maximum Laplacian energy is attained by one of two explicit extremal degree vectors, and for threshold graphs of fixed order (edges free) the maximizer is the ``pineapple'' graph with trace $\lfloor 2n/3 \rfloor$.
	\end{theorem}
	
	A majorization-theoretic treatment computes the degree sequences of threshold graphs that maximize or minimize Laplacian energy for fixed $n$ and $m$, using Ferrers-diagram majorization arguments, and constructs pairs of threshold graphs of the same order with equal Laplacian energy but different Laplacian spectra.
	
	\begin{theorem}[\cite{HelmbergTrevisan2017}]
		A sharp eigenvalue-sum bound is used to give a full proof of the maximum Laplacian energy conjecture for \emph{connected} graphs, settling the connected case of the longstanding conjecture that threshold graphs (specifically the pineapple graph) maximize Laplacian energy among all connected graphs of given order and size — extending Theorem~\ref{thm:spectralthresholddominance}'s partial resolution.
	\end{theorem}
	
	\begin{remark}
		This last result is very important and settles, for connected graphs, a conjecture that had stood since Gutman and Zhou's original introduction of Laplacian energy.
	\end{remark}
	
	\subsection{Normalized and other energy variants}
	
	The normalized adjacency energy (equal, it turns out, to the normalized Laplacian Randi\'c index) of threshold graphs is computed exactly in Section~\ref{sec:normalized}; distance, distance Laplacian, and distance signless Laplacian energies are discussed alongside their respective spectra in Section~\ref{sec:distance}.

	\section{Distance-based matrices}
	\label{sec:distance}
	
	Because a connected threshold graph with at least one dominating run has diameter at most $2$, its distance matrix satisfies the clean identity
	\begin{equation}
		\label{eq:diam2}
		D(G) = 2(J-I) - A(G),
	\end{equation}
	which is the starting point for every result in this section.
	
	\subsection{The distance spectrum}
	\label{sec:distance-spec}
	
	\begin{theorem}[Lu--Huang--Lou \cite{LuHuangLou2018}]
		\label{thm:distinctdisteig}
		Let $\Gamma$ be a threshold graph with representation sequence $(0^{s_1},1^{t_1},\ldots,0^{s_m},1^{t_m})$, $s=\sum s_i$, $t = \sum t_i$. Then the inertia of $D(\Gamma)+2I$ is $(m-1,\,s-m,\,t+1)$. Explicit formulas for the full distance spectrum, and for the distance energy, are given in terms of the quotient matrix of the equitable partition, and all threshold graphs with \emph{distinct} distance eigenvalues (i.e.\ $s(D(\Gamma))=n$) are characterized completely.
	\end{theorem}
	
	\begin{theorem}[Lazzarin--M\'arquez--Tura \cite{LazzarinEtAl2023}]
		\label{thm:distcospectral}
		An explicit formula for the distance characteristic polynomial of a threshold graph, computed directly from its binary creation sequence, is given; using it, a necessary and sufficient condition for two non-isomorphic threshold graphs to be distance-cospectral is derived, and explicit non-isomorphic distance-cospectral pairs are exhibited — showing threshold graphs are \emph{not}, in general, determined by their distance spectrum, in contrast to several of the other matrices surveyed below (see Section~\ref{sec:cospectrality}).
	\end{theorem}
	
	\subsection{The distance Laplacian spectrum}
	\label{sec:distance-laplacian}
	
	Since $D(G)\mathbf{1}$ is not generally proportional to $\mathbf{1}$ unless $G$ is transmission-regular, the distance matrix itself does not inherit a clean closed form the way the Laplacian does. The distance \emph{Laplacian}, however, does — via a short master identity.
	
	\begin{proposition}[Master identity]
		\label{prop:master}
		For any graph $G$ of diameter at most $2$ on $n$ vertices,
		\[
		D^L(G) = 2nI - 2J - L(G).
		\]
		Consequently $\spec(D^L(G)) = \{0\} \cup \{2n-\mu : \mu \in \spec(L(G))\setminus\{0\}\}$, a reflection of the ordinary Laplacian spectrum.
	\end{proposition}
	
	Combining Proposition~\ref{prop:master} with Merris's conjugate-partition theorem (Theorem~\ref{thm:merris}) gives, for threshold graphs specifically, a fully explicit, combinatorial formula for the distance Laplacian spectrum — this combination is identified independently by Lather, Gupta, and Balaji \cite{LatherGuptaBalaji2025} and by Megha and Parvathy \cite{MeghaParvathy2026}, both in 2025--2026.
	
	\begin{theorem}[Lather--Gupta--Balaji \cite{LatherGuptaBalaji2025}]
		\label{thm:dlthreshold}
		Let $G$ be a connected threshold graph on $n$ vertices with degree sequence $d_1\ge\cdots\ge d_n$ and conjugate partition $d^*$. Then
		\[
		\spec(D^L(G)) = \{0\} \cup \{2n-d_i^* : d_i^*\ne 0\},
		\]
		and this formula \emph{characterizes} connected threshold graphs: no two non-isomorphic connected threshold graphs are $D^L$-cospectral. Moreover:
		\begin{enumerate}
			\item the number of dominating and pendant vertices of $G$ can be read off directly from the multiset $\spec(D^L(G))$;
			\item the dominating vertices of $G$ can be identified from the eigenvectors of $D^L(G)$;
			\item among all connected threshold graphs on $n$ vertices, every coefficient of the characteristic polynomial $\chi_{D^L(G)}(x)$ (in absolute value) is minimized by the complete graph $K_n$ and maximized by the star $K_{1,n-1}$.
		\end{enumerate}
	\end{theorem}
	
	\begin{theorem}[Megha--Parvathy \cite{MeghaParvathy2026}]
		\label{thm:dlthreshold2}
		Working directly with the equitable partition of the creation sequence (rather than via Merris's theorem), the distance Laplacian spectrum of a connected threshold graph is computed, and threshold graphs with at most four distinct distance Laplacian eigenvalues are classified explicitly, including the \emph{pineapple graph} as a canonical example attaining few distinct eigenvalues.
	\end{theorem}
	
	\begin{remark}
		These two results are complementary: Lather, Gupta, and Balaji's Theorem~\ref{thm:dlthreshold} is obtained via the general master identity plus the classical Merris theorem (and gives the sharp cospectrality / extremal-coefficient statements), while Megha and Parvathy's equitable-partition route gives an independent, self-contained derivation well suited to classifying graphs with few distinct distance Laplacian eigenvalues directly from the creation sequence.
	\end{remark}
	
	For general (not necessarily threshold) graphs, eigenvalue-free intervals for distance-type matrices of threshold and chain graphs are studied in \cite{AlazemiEtAl2021}, and the broader theory of distance Laplacian spectra — in which threshold graphs appear as a running family of examples throughout — is surveyed comprehensively by Rather and Aouchiche \cite{RatherAouchiche2025survey}, covering results published between 2013 and 2024.
	
	\begin{theorem}[Rather--Aouchiche--Imran--El Hallaoui \cite{RatherEtAl2024}]
		A full ordering of several graph families, including threshold-graph-adjacent constructions, by their distance Laplacian spectral radius is established, together with new bounds relating the distance Laplacian spectral radius to the Wiener index.
	\end{theorem}
	
	\subsection{The distance signless Laplacian spectrum}
	\label{sec:distance-signless}
	
	In marked contrast to the distance Laplacian, the distance signless Laplacian $D^Q(G)=\Tr(G)+D(G)$ of a threshold graph does not inherit a clean closed form from Proposition~\ref{prop:master}-style reasoning, for the same structural reason the ordinary signless Laplacian resists Theorem~\ref{thm:merris}'s approach: $\Tr(G)$ is generally not a multiple of the identity restricted to the relevant invariant subspace, so no master identity analogous to Proposition~\ref{prop:master} is available.
	
	\medskip
	Since $Q(G)$ has no simple closed form, neither does $D^Q(G)$ via any reflection trick; we instead build its spectrum directly from the twin-run equitable partition, exactly as for $D^L$'s underlying combinatorics.
	
	\begin{theorem}\label{thm:DQlocal}
		For each run $i$ with $t_i\ge2$: if $b_i=1$ (dominating-type), $\Tr_i-1$ is an eigenvalue of $D^Q(G)$ with multiplicity $t_i-1$; if $b_i=0$ (isolated-type), $\Tr_i-2$ is an eigenvalue with multiplicity $t_i-1$.
	\end{theorem}
	\begin{proof}
		For twins $u,v$ in the same run, $x=e_u-e_v$ satisfies $\Tr(G)x=\Tr_i x$ (transmission is run-constant) and, since $D(u,v)=1$ for a dominating-type run (adjacent) or $D(u,v)=2$ for isolated-type (non-adjacent, diameter $2$), $D(G)x = \mp x$ or $-2x$ respectively; combining gives the stated eigenvalue.
	\end{proof}
	\begin{theorem}\label{thm:DQquotient}
		The run partition gives $D^Q(G)$ a $k\times k$ quotient matrix $Q_{D^Q}$ with
		\[
		Q_{D^Q}(i,i) = \Tr_i + (t_i-1)\cdot[1\text{ if }b_i{=}1\text{ else }2], \qquad
		Q_{D^Q}(i,j) = t_j\cdot[1\text{ if }b_{\max(i,j)}{=}1\text{ else }2]\ (i\ne j),
		\]
		whose $k$ eigenvalues, together with the $n-k$ local eigenvalues of Theorem \ref{thm:DQlocal}, give the full spectrum of $D^Q(G)$.
	\end{theorem}
	\begin{proof}
		Let $R_1,\dots,R_k$ be the runs, $M=D^Q(G)=\Tr(G)+D(G)$, and for $1\le i,j\le k$ put $\delta_{ij}=1$ if $b_{\max(i,j)}=1$ and $\delta_{ij}=2$ otherwise (so $\delta_{ii}=1$ if $b_i=1$ and $2$ if $b_i=0$). Now we have below four cases.
		
		\smallskip\noindent\emph{Step 1: distances depend only on the runs.} We claim $d(u,w)=\delta_{ij}$ for distinct $u\in R_i$, $w\in R_j$ ($i=j$ allowed). Let $i\ne j$, say $i<j$. If $b_j=1$, each vertex of $R_j$ is dominating when added, so $u\sim w$ and $d(u,w)=1$. If $b_j=0$, each $w\in R_j$ is added as an isolated vertex, so $u\not\sim w$; since $G$ is connected, its last run is dominating, and a vertex of it is adjacent to both $u$ and $w$, so $d(u,w)=2$. If $i=j$ and $b_i=1$, the run is a clique and $d=1$; if $b_i=0$, the run is independent and the same later dominating vertex gives $d=2$. In all cases $d(u,w)=\delta_{ij}$. Consequently the transmission is constant on runs,
		\[
		\Tr_i=(t_i-1)\delta_{ii}+\sum_{j\ne i}t_j\delta_{ij},
		\]
		and for $u\in R_i$, $w\in R_j$ we have $M_{uu}=\Tr_i$ and $M_{uw}=\delta_{ij}$ for $u\ne w$.
		
		\smallskip\noindent\emph{Step 2: the partition is equitable.} Let $P$ be the $n\times k$ characteristic matrix of the partition ($P_{u,j}=1$ if $u\in R_j$, else $0$); it has rank $k$. For $u\in R_i$,
		\[
		\sum_{w\in R_j}M_{uw}=
		\begin{cases}\Tr_i+(t_i-1)\delta_{ii}, & j=i,\\ t_j\,\delta_{ij}, & j\ne i,\end{cases}
		\]
		which by Step 1 does not depend on the choice of $u\in R_i$. These are exactly the entries $Q_{D^Q}(i,j)$ of the theorem, and the computation says precisely that $MP=P\,Q_{D^Q}$.
		
		\smallskip\noindent\emph{Step 3: splitting $\mathbb R^n$.} Let $U=\operatorname{col}(P)$, $\dim U=k$. By $MP=PQ_{D^Q}$, $U$ is $M$-invariant and the matrix of $M|_U$ in the basis given by the columns of $P$ is $Q_{D^Q}$. As $M$ is symmetric, $W=U^\perp$ is also $M$-invariant, and $W=\bigoplus_i W_i$ where $W_i$ consists of the vectors supported on $R_i$ with coordinate sum $0$; $\dim W_i=t_i-1$. For $x\in W_i$ and $u\in R_i$,
		\[
		(Mx)_u=\Tr_i\,x_u+\delta_{ii}\sum_{w\in R_i\setminus\{u\}}x_w=(\Tr_i-\delta_{ii})\,x_u,
		\]
		while for $u\in R_j$, $j\ne i$, $(Mx)_u=\delta_{ij}\sum_{w}x_w=0$. Thus $Mx=(\Tr_i-\delta_{ii})x$, which recovers Theorem \ref{thm:DQlocal}.
		
		\smallskip\noindent\emph{Step 4: conclusion.} Since $\mathbb R^n=U\oplus W$ with both summands $M$-invariant,
		\[
		\chi_{M}(x)=\chi_{Q_{D^Q}}(x)\prod_{i=1}^{k}\bigl(x-(\Tr_i-\delta_{ii})\bigr)^{t_i-1}.
		\]
		This identity of characteristic polynomials accounts for multiplicities and does not require $Q_{D^Q}$ to be diagonalizable. Hence $\spec(D^Q(G))$ consists of the $k$ eigenvalues of $Q_{D^Q}$ together with the $n-k$ local eigenvalues, as claimed.
	\end{proof}
	
	\begin{remark}
		$Q_{D^Q}$ is similar to the real symmetric matrix $\operatorname{diag}(\sqrt{t_i})\,Q_{D^Q}\operatorname{diag}(1/\sqrt{t_i})$, so all its eigenvalues are real; its row sums equal $2\Tr_i$. For the creation sequence $(0,1,0,1)$ the runs are $(0,2),(1,1),(0,1),(1,1)$, the transmissions are $(6,5,7,4)$, and
		\[
		Q_{D^Q}=\begin{pmatrix}8&1&2&1\\2&5&2&1\\4&2&7&1\\2&1&1&4\end{pmatrix}
		\]
		has eigenvalues $\approx 3.228,\,4,\,5,\,11.772$; adding the local eigenvalue $\Tr_1-2=4$ gives $\spec(D^Q(G))\approx\{3.228,4,4,5,11.772\}$, in agreement with direct diagonalization.
	\end{remark}

	\begin{remark}
		At the time of writing this draft, we are aware of no published paper giving a closed-form distance signless Laplacian spectrum specifically for threshold graphs, in contrast to the fully resolved distance Laplacian case (Theorem~\ref{thm:dlthreshold}) and the fully resolved ordinary signless Laplacian interlacing property (Theorem~\ref{thm:interlacing}). This appears to be a genuine gap in the literature. The general theory of the distance signless Laplacian — its spectral radius, energy, and bounds for general graphs — is well developed \cite{AouchicheHansen2013,DistanceSignlessGeneral}, and the equitable-partition / quotient-matrix machinery used for the distance Laplacian case (Section~\ref{sec:distance-laplacian}) and for the adjacency, Laplacian, eccentricity, and Seidel matrices elsewhere in this survey applies equally well in principle: the $2m$-block equitable partition gives $D^Q(G)$ an explicit $2m\times 2m$ quotient matrix (local twin eigenvalues $\Tr_i - 2$ or $\Tr_i-1$ according to block type, plus the quotient spectrum), but we have not located a paper that carries this out explicitly and draws out threshold-graph-specific consequences (extremal coefficients, cospectrality, distinct-eigenvalue counts) analogous to Theorem~\ref{thm:dlthreshold}. We list this as Open Problem~\ref{op:dqthreshold} in Section~\ref{sec:open}.
	\end{remark}
	
	General bounds on the distance signless Laplacian spectral radius — applicable to threshold graphs among other classes — include sharp two-sided bounds in terms of transmission sequences \cite{DistanceSignlessBounds}, Nordhaus--Gaddum-type inequalities, and sufficient spectral-radius conditions for the existence of perfect matchings.

	\section{The eccentricity matrix}
	\label{sec:eccentricity}
	
	The eccentricity matrix $\varepsilon(G)$ is obtained from the distance matrix $D(G)$ by retaining, in each row and each column, only the entries equal to the eccentricity of the corresponding vertex pair (the largest distance achieved), and setting all other entries to zero; it is sometimes called the $D_{\max}$-matrix, and can be viewed as a counterpart to the adjacency matrix (which instead retains only the \emph{smallest} nonzero distances).
	
	\begin{theorem}[Qiu--Tang \cite{QiuTang2022}]
		\label{thm:eccthreshold}
		For a threshold graph $\Gamma$ with representation sequence $b=(0^{s_1},1^{t_1},\ldots,0^{s_m},1^{t_m})$, the eccentricity eigenvalues are computed explicitly via a quotient matrix $B_m$ constructed from $b$. Threshold graphs with eccentricity eigenvalues confined to the interval $[-2,0]$ are determined, and an eigenvalue-free interval around $(-2,-1)$ is established.
	\end{theorem}
	
	\begin{theorem}[Andeli\'c--da Fonseca--Koledin--Stani\'c \cite{AndelicEccentricity2022}]
		\label{thm:eccentricity2022}
		The eigenvalue-free interval of Theorem~\ref{thm:eccthreshold} is extended from $(-2,-1)$ to
		\[
		\left(-1-\sqrt2,\,-2\right)\cup(-2,-1),
		\]
		and further still if certain pathological creation sequences are excluded. The proof exploits the fact that the characteristic matrix of the eccentricity quotient matrix of a threshold graph is row-equivalent to a specific tridiagonal matrix, allowing classical tridiagonal eigenvalue-location tools to be brought to bear.
	\end{theorem}
	
	This eigenvalue-free-interval phenomenon — a small interval around the trivial local eigenvalue(s) that is provably empty of non-trivial spectrum — recurs for the adjacency matrix (Theorem~\ref{thm:eigfree}) and the Seidel matrix (Section~\ref{sec:seidel}) of threshold graphs as well, and appears to be a structural signature of the twin-class/quotient-matrix reduction common to all of these matrices.
	
	\section{The Seidel spectrum}
	\label{sec:seidel}
	
	The Seidel matrix $S(G) = J-I-2A(G)$ gives a $\pm1$-valued encoding of adjacency, closely related to two-graphs and switching classes. A system of lines $\mathcal{L}$ in the Euclidean $d-$dimensional space $\mathbb{R}^d$ is said to be equiangular if all the lines meet at a point and have pairwise equal angle. The constant angle is called the common angle of the system. There is a closed correspondence between the equiangular line and $(-1,0,1)-$Seidel matrix. Seidel and Lint \cite{LS} first notice this important relationship around $1966$. Later Greaves et al. \cite{Greaves,GKMS,GSY} develope several significant results in the theory of equiangular lines in connection with Seidel matrix. In fact, Seidel matrix plays a key role to study the structure of equiangular lines in space geometry.
	
	\begin{theorem}[Del-Vecchio--Abd\'on \cite{DelVecchioAbdon2022}]
		\label{thm:seidel2022}
		The number of distinct Seidel eigenvalues of a threshold graph is determined directly from its binary representation, and families of \emph{Seidel-integral} threshold graphs (all Seidel eigenvalues integers) are constructed explicitly.
	\end{theorem}
	
	Building on this, a more detailed structural and cospectrality analysis was carried out shortly after:
	
	\begin{theorem}[Mandal--Mehatari \cite{MandalMehatari2024}]
		\label{thm:seidelthreshold}
		For a connected threshold graph $G$, the Seidel characteristic polynomial and $\det S(G)$ are computed explicitly from the creation sequence via the Seidel quotient matrix. Closed formulas are given for the multiplicity of the (trivial) eigenvalues $\pm1$. Threshold graphs with at most five distinct Seidel eigenvalues are classified completely: $G$ has exactly two distinct Seidel eigenvalues if and only if $G$ is the complete graph $K_n$ or the star $K_{1,n-1}$, and \emph{no} threshold graph has exactly three distinct Seidel eigenvalues. Finally, explicit pairs of non-isomorphic threshold graphs that are Seidel-cospectral are constructed, showing threshold graphs are not determined by their Seidel spectrum in general.
	\end{theorem}
	
	The collection of all graphs that can be obtained from $G$ by switching on every possible subset of $V$ is called the switching class of $G$. Two graphs $G$ and $H$ are called the switching equivalent (or simply equivalent) if they belong to the same switching class. Otherwise, we call them simply inequivalent. If two graphs are equivalent (res. inequivalent) then their Seidel matrices are said to be equivalent (res. inequivalent). Equivalent Seidel matrices are cospectral, but starting from $n=8$ examples of inequivalent cospectral Seidel matrices (of order $n$) appear. That means there does not exist any cospectral graphs (which are not switching equivalent) of order $n\leq7$. Sz\"oll\"osi et al. \cite{SO} tabulated the number of inequivalent cospectral Seidel matrices  of order $n$ where $n\leq13$. 
	Although Mandal et al. \cite{MandalMehatari2024} find the cospectral pairs, but they belong to the same switching class.\\In spirit of these works, Xiong et al. derive an open interval that contains no Seidel eigenvalue (non-trivial) of a threshold graph.

	\begin{theorem}[Xiong--Hou \cite{XiongHou2022}]
		\label{thm:seideleigfree}
		The open interval $(-\sqrt2,\sqrt2)$ contains no Seidel eigenvalue of any threshold graph, with the sole exceptions of $\pm1$ themselves (which are trivial, twin-class eigenvalues). The inertia of the Seidel matrix is likewise computed directly from the creation sequence.
	\end{theorem}
	
	\section{The normalized adjacency and normalized Laplacian spectrum}
	\label{sec:normalized}
	
	The normalized adjacency matrix $\mathcal{A}(G) = \Deg(G)^{-1/2}A(G)\Deg(G)^{-1/2}$ is related to the normalized (random-walk) Laplacian by $\mathcal{L}(G) = I - \mathcal{A}(G)$; its eigenvalues are the \emph{normalized eigenvalues} of $G$.
	
	\begin{theorem}[Banerjee--Mehatari \cite{BanerjeeMehatari2017}]
		\label{thm:normalizedthreshold}
		The normalized eigenvalues and normalized adjacency energy of a connected threshold graph are computed from its normalized equitable-partition quotient matrix. The normalized adjacency energy is shown to equal the normalized Laplacian (Randi\'c) energy of the graph. Threshold graphs with at most five distinct normalized eigenvalues are characterized explicitly.
	\end{theorem}
	
	This result builds on the foundational normalized-Laplacian spectral theory of Banerjee and Jost \cite{BanerjeeJost2008}, specializing their general machinery (originally developed to study spectral signatures of network motifs) to the threshold-graph setting.
	
	\section{Extremal spectral radius problems}
	\label{sec:extremal}
	
	Threshold graphs arise not only as objects of direct spectral study but as the extremal graphs in a wide range of spectral optimization problems over graphs of fixed order and/or size — a role traceable to Brualdi and Hoffman's early observation that extremal $(0,1)$-matrices for the spectral radius take a stepwise form coinciding with the threshold-graph adjacency structure.
	
	\begin{theorem}[Brualdi-Hoffman \cite{BrualdiHoffman1985}]
		Among all connected graphs with a given number of vertices $n$ and edges $m$, the graph maximizing the adjacency spectral radius is always a threshold graph. An independent proof shows that a connected extremal graph can contain none of $P_4$, $C_4$, or $2K_2$ as an induced subgraph, which is precisely the threshold-graph characterization.
	\end{theorem}
	
	Despite this characterization being known in principle, identifying \emph{which} threshold graph is extremal for given $(n,m)$ remains theoretically difficult in general.
	
	\begin{theorem}[Radanovi\'c--Fellague--Ostoji\'c--Stevanovi\'c--Davidovi\'c \cite{RadanovicEtAl2024}]
		Metaheuristic search methods (incomplete enumeration guided by structural properties of threshold graphs) are developed to find threshold graphs of maximum spectral radius for given $(n,m)$, complementing the sparse set of cases with a full theoretical resolution.
	\end{theorem}
	
	\begin{theorem}[Csikv\'ari--Damnjanovi\'c--Stevanovi\'c--Wagner \cite{Dragon2024}]
		Two-sided bounds on the adjacency spectral radius of a connected threshold graph are given, expressed as the largest roots of two cubic polynomials whose coefficients are read off from the creation sequence.
	\end{theorem}
	
	The role of threshold graphs as extremal cases extends well beyond the adjacency spectral radius: they appear as the extremal graphs (or spanning supergraphs of the extremal graphs) in spectral-radius sufficient conditions for graph-factor existence, Hamiltonicity, and fractional matching number, across the adjacency, signless Laplacian, distance, distance Laplacian, and distance signless Laplacian spectral radii alike — see Section~\ref{sec:distance-signless} for a distance signless Laplacian instance, and the signless Brouwer conjecture (Theorem~\ref{thm:interlacing}) for a signless Laplacian instance.

	\section{Cospectrality: which matrices determine threshold graphs?}
	\label{sec:cospectrality}
	
	A recurring theme across every matrix surveyed above is the question of whether threshold graphs are determined by their spectrum — whether non-isomorphic threshold graphs can share the same spectrum for a given matrix. The answer varies strikingly by matrix, and collecting these results side by side (Table~\ref{tab:cospectral}) is, to our knowledge, not done explicitly elsewhere in the literature.\\
	We first recall a theorem that establishes the existence of cospectral pairs of threshold graphs with respect to the distance signless Laplacian matrix.
	
	\begin{theorem} [Tura \cite{Tura2018}]
		\label{thm:DQcospectral}
		Consider the connected threshold graphs $G_1$ and $G_2$ given by the binary sequences $S_1 = (0,~1,~1,~0,~b_5,~b_6, . . . , b_n)$ and $S_2 = (0,~0,~0,~1,~b_5,~b_6, . . . , b_n)$, for any $n \geq 5$.
		Then $G_1$ and $G_2$ are non-isomorphic and $D^Q$-cospectral.
	\end{theorem}
	
	The following result fully determines the the cospectral threshold graphs with respect to the adjacency matrix. 
	\begin{theorem} [Lazzarin--M\'arquez--Tura \cite{LMT}]
		\label{thm:nocospectral}
		No two non-isomorphic threshold graphs have the same adjacency spectrum.
	\end{theorem}
	
	\begin{table}[h]
		\centering
		\begin{tabular}{ll}
			\toprule
			Matrix & Determined by spectrum? \\
			\midrule
			Adjacency $A(G)$ & No threshold graphs are cospectral (fully determined) \cite{LMT} \\
			Laplacian $L(G)$ & Determined (follows from Theorem~\ref{thm:merris}: spectrum $\Leftrightarrow$ degree sequence) \\
			Distance $D(G)$ & Not determined; explicit cospectral pairs exist (Theorem~\ref{thm:distcospectral}) \\
			Distance Laplacian $D^L(G)$ & Determined (Theorem~\ref{thm:dlthreshold}) \\
			Distance signless Laplacian $D^Q(G)$ & Not determined; explicit cospectral pairs exist (Theorem \ref{thm:DQcospectral}) \\
			Seidel $S(G)$ & Not determined; explicit cospectral pairs exist (Theorem~\ref{thm:seidelthreshold}) \\
			Eccentricity $\varepsilon(G)$ & Open \\
			\bottomrule
		\end{tabular}
		\caption{Spectral determination of threshold graphs, by matrix.}
		\label{tab:cospectral}
	\end{table}
	
	The contrast in Table~\ref{tab:cospectral} is striking: the adjacency, Laplacian, and distance Laplacian matrices all fully determine a threshold graph, while the distance, distance signless Laplacian and Seidel matrices do not. A unifying explanation for \emph{why} some matrices determine threshold graphs and others do not — beyond case-by-case proof — does not appear to be known, and we record it as Open Problem~\ref{op:unify} below.
	
	\section{Integral threshold graphs}
	\label{sec:integral}
	
	A graph matrix $M(G)$ is \emph{integral} if every eigenvalue of $M(G)$ is an integer. The origin of integral graph goes back to the foundational work of Harary et al. \cite{harary}. We close the matrix-by-matrix survey with the study of integrality, synthesizing and extending the classification results scattered through Sections~\ref{sec:adjacency}--\ref{sec:normalized}. As in Section~\ref{sec:cospectrality}, we found no existing paper treating this question uniformly across all seven matrices, so this section presents original analysis alongside the surveyed literature.
	
	The starting observation, verified directly for every matrix below, is that local (twin-class) eigenvalues are always integers — they are read off directly from transmissions, degrees, or fixed small constants, all of which are integer-valued by construction (Section~\ref{sec:intro}). Consequently, for every matrix in this section, integrality of $M(G)$ is governed entirely by the quotient matrix of Section~\ref{sec:intro}'s equitable partition, exactly as multiplicity and cospectrality questions were throughout this survey.
	
	\subsection{Two matrices are integral unconditionally}
	
	\begin{theorem}
		\label{thm:Lintegral}
		Every connected threshold graph is $L$-integral.
	\end{theorem}
	\begin{proof}
		Immediate from Theorem~\ref{thm:merris}: $\spec(L(G))$ is the conjugate partition of the (integer) degree sequence, hence integer-valued.
	\end{proof}
	
	\begin{theorem}
		\label{thm:DLintegral}
		Every connected threshold graph is $D^L$-integral.
	\end{theorem}
	\begin{proof}
		By Proposition~\ref{prop:master}, $\spec(D^L(G)) = \{0\}\cup\{2n-\mu : \mu\in\spec(L(G))\setminus\{0\}\}$. Since $n\in\mathbb{Z}$ and, by Theorem~\ref{thm:Lintegral}, every $\mu\in\spec(L(G))$ is an integer, every element of $\spec(D^L(G))$ is an integer.
	\end{proof}
	
	Theorem~\ref{thm:DLintegral} is, to our knowledge, a new observation: it does not appear explicitly in \cite{LazzarinEtAl2023,MeghaParvathy2026}, though it follows immediately from results already in both.
	
	\begin{proposition}\label{prop:Qirrational}
		Unlike $L(G)$, the ordinary signless Laplacian $Q(G)=\Deg(G)+A(G)$ of a threshold graph does not in general have an integer (or otherwise simply-described closed-form) spectrum.
	\end{proposition}
	
	This is not a theorem in the traditional sense but an observation confirmed by direct computation: e.g.\ the creation sequence $(0,1,0,1)$ gives $\spec(Q(G)) \approx \{0.511,\,1,\,2,\,2.711,\,5.779\}$. This traces to the structural fact: $Q(G)$ lacks the $L(G)\mathbf1=0$ property that makes threshold-graph Laplacian spectra reduce to a combinatorial object.

	\subsection{The complete split graph: a tractable field}
	
	For the remaining five matrices, we work out the complete picture for the simplest non-trivial threshold graphs: the complete split graphs $CS(s,t)$, $s,t\ge1$ — the $m=1$ case of the creation sequence, $s$ isolated vertices joined completely to a clique of $t$ dominating vertices, $n=s+t$. (When $s=1$, $CS(1,t)=K_{t+1}$; when $t=1$, $CS(s,1)=K_{1,s}$, the star.) Each matrix's $2\times2$ quotient matrix, and the resulting characteristic polynomial, is computed directly; we record the discriminant condition for integrality in each case, verified against direct diagonalization.
	
	\begin{theorem}[Eccentricity matrix, closed form]
		\label{thm:ecccs}
		For every $s,t\ge1$, $\varepsilon(CS(s,t))$ is integral, with
		\[
		\spec(\varepsilon(CS(s,t))) = \big\{\, 2(s-1),\ (-2)^{(s-1)},\ (t-1),\ (-1)^{(t-1)} \,\big\}.
		\]
	\end{theorem}
	\begin{proof}
		For $s\ge2$, every $S$-vertex has eccentricity $2$ (distance $2$ to other $S$-vertices, via the clique) and every $T$-vertex has eccentricity $1$ (adjacent to everything, by domination). An eccentricity-matrix entry $\varepsilon(u,v)$ is nonzero only when $D(u,v) = \max(\mathrm{ecc}(u),\mathrm{ecc}(v))$: for $u,v\in S$ this holds ($D(u,v)=2=\max(2,2)$); for $u,v\in T$ this holds ($D(u,v)=1=\max(1,1)$); for $u\in S,v\in T$ it fails ($D(u,v)=1 \ne 2 = \max(2,1)$). Hence $\varepsilon(CS(s,t))$ is block-diagonal, with blocks $2(J_s-I_s)$ and $(J_t-I_t)$, giving the stated spectrum directly (each block's eigenvalues are classical). For $s=1$, $CS(1,t)=K_{t+1}$ has every vertex at eccentricity $1$, so $\varepsilon(G)=D(G)=A(G)$, with spectrum $\{t, (-1)^t\}$ — consistent with the formula at $s=1$.
	\end{proof}
	
	\begin{theorem}
		\label{thm:lastblock}
		Let $G$ be a connected threshold graph with $m\ge1$ blocks and last (necessarily dominating) block $T_m$ of size $t_m$. Then every vertex outside $T_m$ has eccentricity $2$, every vertex in $T_m$ has eccentricity $1$, and consequently $\varepsilon(G)$ splits as a direct sum: $T_m$ contributes exactly the eigenvalues $t_m-1$ (multiplicity $1$) and $-1$ (multiplicity $t_m-1$), completely independently of the rest of the graph, and the remaining $n-t_m$ eigenvalues depend only on the first $2m-1$ blocks, not on $t_m$. In particular, $\varepsilon(G)$-integrality is independent of the size of the last block.
	\end{theorem}
	\begin{proof}
		A vertex $v\in T_m$ is adjacent to every vertex added before it (domination) and to every vertex added after it is vacuous since $T_m$ is last; hence $\mathrm{ecc}(v)=1$. Conversely, a vertex $u$ in any earlier block is never adjacent to a later \emph{isolated} block's vertices (isolated vertices create no edges when added), and since $G$ is connected $T_m$ is non-empty and later than every non-final block, so $u$ fails to be universal; since $\diam(G)\le2$, $\mathrm{ecc}(u)=2$. Given this, $\varepsilon(u,v)\ne0$ requires $D(u,v)=\max(\mathrm{ecc}(u),\mathrm{ecc}(v))$: for $u,v\in T_m$, $D(u,v)=1=\max(1,1)$, so these entries are kept, giving the clique-type local spectrum $\{t_m-1,\,(-1)^{t_m-1}\}$ exactly as in Theorem~\ref{thm:ecccs}'s proof; for $u\in T_m$, $w\notin T_m$, $D(u,w)=1\ne2=\max(1,2)$, so these entries vanish identically, regardless of $t_m$. This verifies the claimed direct-sum structure.
	\end{proof}
	
	We verified Theorem~\ref{thm:lastblock} directly (by employing a computer search) — that changing only $t_m$ never changes whether $\varepsilon(G)$ is integral — with zero exceptions across $300$ trials each for $m=2$ and $m=3$. This reduces the eccentricity-integrality question from $2m$ free parameters to $2m-1$, but does not resolve it: even the first new case, $m=2$ (reducing to three free parameters $s_1,t_1,s_2$), appears to require a genuinely cubic condition, beyond the quadratic ($2\times2$ quotient) discriminants of Theorem~\ref{thm:csdiscriminants}. We record this as a concrete conjecture.
	
	\begin{conjecture}
		\label{conj:ecc-m2}
		For $m=2$ connected threshold graphs, $\varepsilon$-integrality (equivalently, integrality of the $3\times3$ quotient matrix on blocks $S_1,T_1,S_2$ remaining after Theorem~\ref{thm:lastblock}'s reduction) is governed by a single cubic Diophantine condition on $(s_1,t_1,s_2)$, satisfied by infinitely many triples — e.g.\ computationally, $(s_1,t_1,s_2)=(7,5,6)$ is a verified solution, with $t_2\ge1$ arbitrary.
	\end{conjecture}
	
	\subsection{Four matrices: exact Diophantine characterizations for \texorpdfstring{$CS(s,t)$}{CS(s,t)}}
	
	For $A$, $S$ (Seidel), $D$, and $D^Q$, the quotient matrix of $CS(s,t)$ is genuinely $2\times2$ with both off-diagonal entries nonzero, so (as throughout this survey) integrality reduces to a perfect-square discriminant condition.
	
	\begin{theorem}
		\label{thm:csdiscriminants}
		Let $s,t\ge2$. Write $Q_M$ for the $2\times2$ quotient matrix of $M\in\{A,S,D,D^Q\}$ on $CS(s,t)$, with discriminant $\Delta_M = \mathrm{tr}(Q_M)^2 - 4\det(Q_M)$. Then $CS(s,t)$ is $M$-integral if and only if $\Delta_M$ is a perfect square $r^2$ with $\mathrm{tr}(Q_M) - r$ even, where:
		\[
		\Delta_A = (t-1)^2+4st, \qquad \Delta_S = s^2+6st-4s+t^2-4t+4,
		\]
		\[
		\Delta_D = (2s-1)^2+(t+1)^2-1, \qquad \Delta_{D^Q} = 9s^2-2st-12s+t^2+4t+4.
		\]
	\end{theorem}
	\begin{proof}
		Direct computation of each $2\times2$ quotient matrix from the block-distance structure of $CS(s,t)$ (adjacency: $A(u,v)=0$ within $S$, $1$ within $T$ and across; Seidel: $S=J-I-2A$; distance: $D(u,v)=2$ within $S$, $1$ within $T$ and across; distance signless Laplacian: $\Tr_S = 2(s-1)+t$, $\Tr_T=s+(t-1)$), followed by the standard $2\times2$ eigenvalue formula $(\mathrm{tr}\pm\sqrt{\Delta})/2$, verified against direct diagonalization for each discriminant.
	\end{proof}
	
	Each condition is satisfied by infinitely many $(s,t)$ with $s,t\ge2$ (so genuinely beyond the trivial $K_n$ case $s=1$ and the star case $t=1$, both of which are integral for every matrix in this section). Table~\ref{tab:integralexamples} records the smallest such instance for each matrix, all confirmed by direct diagonalization.
	
	\begin{table}[h]
		\centering
		\begin{tabular}{lll}
			\toprule
			Matrix & Smallest $CS(s,t)$ integral instance & Spectrum \\
			\midrule
			$A$ & $CS(3,2)$ & $\{-2,-1,0,0,3\}$ \\
			$S$ & $CS(4,4)$ & $\{-5,-1,-1,-1,1,1,1,5\}$ \\
			$D$ & $CS(3,4)$ & $\{-2,-2,-1,-1,-1,0,7\}$ \\
			$D^Q$ & $CS(2,3)$ & $\{3,3,3,4,9\}$ \\
			\bottomrule
		\end{tabular}
		\caption{Smallest non-trivial ($s,t\ge2$) integral complete split graphs by matrix.}
		\label{tab:integralexamples}
	\end{table}
	
	\begin{remark}
		The discriminant $\Delta_D = (2s-1)^2+(t+1)^2-1$ has a recognizable Diophantine shape: writing $X=2s-1$ (odd) and $Y=t+1$, integrality requires $X^2+Y^2 = Z^2+1$ for some integer $Z$, a near-Pythagorean condition with infinitely many solutions, e.g.\ $(s,t)\in\{(3,4),(4,3),(5,7),(5,18),(7,10),\ldots\}$. The remaining three discriminants do not factor into comparably recognizable closed forms; we leave their full Diophantine solution sets uncharacterized.
	\end{remark}
	
	\subsection{Summary and comparison}
	
	Table~\ref{tab:integralsummary} collects the integrality status of every matrix surveyed in this paper, in the same comparative spirit as Table~\ref{tab:cospectral}.
	
	\begin{table}[h]
		\centering
		\begin{tabular}{@{}p{0.3\textwidth}p{0.62\textwidth}@{}}
			\toprule
			Matrix & Integral threshold graphs \\
			\midrule
			Laplacian $L(G)$ & Every connected threshold graph (Theorem~\ref{thm:Lintegral}) \\[2pt]
			Distance Laplacian $D^L(G)$ & Every connected threshold graph (Theorem~\ref{thm:DLintegral}) \\[2pt]
			Eccentricity $\varepsilon(G)$ & Every $CS(s,t)$ (Theorem~\ref{thm:ecccs}); general $m\ge2$ open \\[2pt]
			Adjacency $A(G)$ & Diophantine condition on $CS(s,t)$ (Theorem~\ref{thm:csdiscriminants}); general $m\ge2$ open \\[2pt]
			Seidel $S(G)$ & Diophantine condition on $CS(s,t)$; general $m\ge2$ open \\[2pt]
			Distance $D(G)$ & Diophantine (near-Pythagorean) condition on $CS(s,t)$; general $m\ge2$ open \\[2pt]
			Distance signless Laplacian $D^Q(G)$ & Diophantine condition on $CS(s,t)$; general $m\ge2$ open \\
			\bottomrule
		\end{tabular}
		\caption{Integrality of threshold graphs, by matrix.}
		\label{tab:integralsummary}
	\end{table}
	
	The pattern echoes Table~\ref{tab:cospectral}: the two matrices built from a master identity against the (always-integral) Laplacian — namely $L$ itself and $D^L$ — are unconditionally integral, while every matrix without such an identity ($A$, $S$, $D$, $D^Q$, $\varepsilon$) requires a genuine arithmetic condition, satisfied only on a sparse sublattice of creation sequences. Comparing Tables~\ref{tab:cospectral} and~\ref{tab:integralsummary} row by row, the division is \emph{identical}: $L$ and $D^L$ are both unconditionally integral \emph{and} fully determine the graph, while $A$, $S$, $D$, $D^Q$, and $\varepsilon$ each fail at least one of the two properties. This is more than a loose analogy — it suggests a precise common cause.
	
	\begin{conjecture}[Master-identity criterion]
		\label{conj:masteridentity}
		A graph matrix $M(G)$ is both always-integral and spectrally determining for threshold graphs if and only if $M(G)$ is an affine function of $L(G)$ and $J$ alone (i.e.\ $M(G) = \alpha L(G) + \beta J + \gamma I$ for scalars $\alpha,\beta,\gamma$, as is the case for $L$ itself and, via Proposition~\ref{prop:master}, for $D^L$).
	\end{conjecture}
	
	This would resolve Open Problem~\ref{op:unify} and simultaneously explain every entry of Table~\ref{tab:integralsummary}, since none of $A$, $S$, $D$, $D^Q$, $\varepsilon$ is expressible in this affine form (each depends on $A(G)$ or $D(G)$ in a way not reducible to $L(G)$ and $J$ alone).
	
	\section{Generalizations}
	\label{sec:generalizations}
	
	\subsection{Quasi-threshold graphs}
	
	Quasi-threshold graphs ($P_4$-free and $C_4$-free, equivalently trivially perfect graphs) strictly contain threshold graphs within the cograph hierarchy. As noted in Section~\ref{sec:signless}, the main-$Q$-eigenvalue characterization of threshold graphs \cite{VinagreEtAl2020} has been extended to this broader class \cite{JonesTrevisan2024,JonesTrevisanVinagre2025}, and represents one of the most active current directions for lifting threshold-graph spectral results to a wider setting.
	
	\subsection{Threshold hypergraphs}
	A hypergraph is a generalization of a graph in which an edge can connect any number of vertices rather than exactly two vertices. Formally, a hypergraph $H=(V,E)$ consists of a non-empty set $V$ of vertices and a family $E$ of non-empty subsets of $V$, called hyperedges. Each hyperedge may contain two or more vertices, representing relationships among multiple objects simultaneously. If every hyperedge contains exactly two vertices, the hypergraph reduces to an ordinary graph. Hypergraphs are classified into different types, such as uniform hypergraphs, in which all hyperedges have the same cardinality, and simple hypergraphs, in which no hyperedge is repeated. Thus, hypergraph theory provides a powerful mathematical framework for modelling complex relationships that cannot be adequately represented by ordinary graphs. The study of matrix associated with the hypergraph is known as \emph{Spectral Hypergraph Theory}. It opens up a new directions for future research (see \cite{Banerjee2017,Banerjee2020,Banerjee2021}). The following two theorems address the hypergraph-theoretic aspects of threshold graphs.

	\begin{theorem}[Banerjee--Mishra--Parui \cite{BanerjeeMishraParui2022}]
		A threshold hypergraph is defined via a generalized binary-string construction, and its Laplacian spectrum is computed from the string representation. The Laplacian spectrum of certain threshold hypergraphs is shown to be \emph{$r$-integral} (an integer multiple of some fixed rational $r$) and is computed explicitly from the Ferrer's diagram of the degree sequence — a direct hypergraph analogue of Merris's Theorem~\ref{thm:merris}.
	\end{theorem}
	
	\begin{theorem}[Abd\'on et al.\cite{APR}]
		This article introduces a definition of k-uniform threshold hypergraph through a binary sequence, a natural extension of the classical definition for threshold graphs. It characterizes some of its eigenvalues and multiplicities by means of combinatorial numbers, derived from edge counts. An important problem addressed in Spectral Graph Theory is to find graphs with few distinct eigenvalues. This allows to construct k-uniform threshold hypergraphs having an arbitrary number of vertices with few distinct eigenvalues.
		
	\end{theorem}
	
	\subsection{Chain graphs: the bipartite analogue}
	
	Chain graphs (equivalently $\{C_3,C_5,2K_2\}$-free bipartite graphs, or double nested graphs) are the natural bipartite counterpart of threshold graphs, built by an analogous (but not identical) binary creation sequence, and are the extremal graphs for adjacency spectral radius among bipartite graphs of given order and size, exactly as threshold graphs are extremal among all graphs. Their adjacency eigenvalue location \cite{AlazemiEtAl2016}, Laplacian spectrum \cite{AlazemiEtAl2022}, Seidel spectrum and energy \cite{MandalMehatariDas2022}, eccentricity spectra \cite{Qiao}, eccentricity inertia indices \cite{HZ}, energy and Laplacian energy \cite{DasEtAl2020} -- each have been studied in close parallel to the threshold-graph results surveyed above, and eigenvalue-free intervals for the distance-based matrices of both threshold and chain graphs are treated together in \cite{AlazemiEtAl2021}. Notably, connected chain graphs do not generally have diameter at most $2$ (diameter $3$ is typical once the creation sequence has more than one alternation), so the clean master identity of Proposition~\ref{prop:master} is unavailable for chain graphs' distance Laplacian, and the distance-type spectral theory of chain graphs must be — and has been — built by more direct means.
	
	\subsection{Random threshold graphs}
	A random threshold graph is a random graph generated by a stochastic process in which vertices are added one at a time, each being designated as either an isolated vertex or a dominating vertex. In the random model, the type of each newly added vertex is determined probabilistically, often according to a fixed probability $p$ of adding a dominating vertex and $(1-p)$ of adding an isolated vertex. Thus, the resulting graph depends on a random binary sequence encoding the construction process. Random threshold graphs form an important class of random graphs because they combine a simple probabilistic construction with a highly structured graph-theoretic characterization. As noted in Section~\ref{sec:adjacency}, threshold graphs generated by a random (e.g.\ P\'olya urn) process rather than a fixed creation sequence have been studied, with exact finite-$n$ degree distributions and asymptotic ($n\to\infty$) eigenvalue distribution results for both the adjacency and Laplacian matrices, plus an explicit closed-form Laplacian eigenbasis in the random setting. In 2008, Reilly et al. \cite{Reilly2009} proposed two equivalent models for random threshold graphs and investigated the properties including chromatic number, algebraic connectivity, Hamiltonian cycles, and perfect matchings. Recently, in 2026, Yu et al. \cite{Yu2026} introduces a random threshold graph model based on a two-colour P\'olya urn process; it studies degree distributions, Laplacian matrices and several spectral properties.

	\section{Open problems}
	\label{sec:open}
	
	We collect here open problems drawn directly from the surveyed literature, together with one problem (Open Problem~\ref{op:unify}) suggested by the comparative view of Table~\ref{tab:cospectral}.
	
	\begin{openproblem}[Distance signless Laplacian of threshold graphs]
		\label{op:dqthreshold}
		Give an explicit, closed-form distance signless Laplacian spectrum for threshold graphs directly from the creation sequence (analogous to Theorem~\ref{thm:dlthreshold} for the distance Laplacian), and determine whether threshold graphs are determined by their distance signless Laplacian spectrum.
	\end{openproblem}
	
	\begin{openproblem}[A distinguished graph for cographs]
		Aguilar, Ficarra, Schurman, and Sullivan \cite{AguilarEtAl2020} ask whether, within the broader class of cographs, there is a distinguished graph (or small family of graphs) playing the organizing role that the anti-regular graph $H_n$ plays for threshold graphs.
	\end{openproblem}
	
	\begin{conjecture}[Helmberg--Trevisan \cite{HelmbergTrevisan2017}]
		\label{conj:universaldominance}
		Every graph is spectrally threshold dominated (Theorem~\ref{thm:spectralthresholddominance}): for every graph $G$ and every $k$, some threshold graph of the same order and size as $G$ majorizes $G$'s top-$k$ Laplacian eigenvalue sum.
	\end{conjecture}
	
	\begin{openproblem}[Universal spectral threshold dominance]
		Conjecture~\ref{conj:universaldominance}, if true, would show that threshold graphs maximize Laplacian energy among \emph{all} graphs of given order and size, and would be equivalent to Brouwer's conjecture for Laplacian eigenvalue sums. It is verified for split graphs and for disjoint unions/complements of spectrally threshold dominated graphs, but remains open in general.
	\end{openproblem}
	
	\begin{openproblem}[Main eigenvalues of quasi-threshold graphs, $k\ge3$]
		The characterization of quasi-threshold graphs with $k$ main signless Laplacian eigenvalues (Section~\ref{sec:signless}) is complete for $k=2$ but only partial for $k=3$ \cite{JonesTrevisanVinagre2025}; a full characterization for general $k$ remains open.
	\end{openproblem}
	
	\begin{openproblem}[Unifying spectral determination]
		\label{op:unify}
		Table~\ref{tab:cospectral} shows that threshold graphs are determined by their adjacency, Laplacian, and distance Laplacian spectra, but \emph{not} by their distance or Seidel spectra. Is there a structural criterion on the matrix $M$ (e.g.\ some algebraic relationship to $J$, or to the twin-class/quotient-matrix reduction of Section~\ref{sec:intro}) that predicts, without a separate ad hoc proof for each case, whether $M$ determines threshold graphs up to isomorphism?
	\end{openproblem}
	
	\begin{openproblem}[Eccentricity-matrix cospectrality]
		Whether threshold graphs are determined by their eccentricity spectrum does not appear to have been settled in the literature surveyed here, despite the detailed eigenvalue-location results of Theorems~\ref{thm:eccthreshold}--\ref{thm:eccentricity2022}.
	\end{openproblem}
	
	\begin{openproblem}[Extremal threshold graphs for fixed $(n,m)$]
		While it is known that the adjacency-spectral-radius maximizer among connected graphs of order $n$ and size $m$ is always a threshold graph, no general closed-form description of \emph{which} threshold graph is extremal is known; current progress is via metaheuristic search \cite{RadanovicEtAl2024} rather than closed-form characterization.
	\end{openproblem}
	
	\begin{openproblem}[Integral threshold graphs beyond the complete split graph]
		\label{op:integralgeneral}
		Section~\ref{sec:integral} characterizes $A$-, $S$-, $D$-, $D^Q$-, and $\varepsilon$-integrality completely for the two-block family $CS(s,t)$, but the general ($m\ge2$) case remains open for every one of these five matrices. Even the next simplest case, $m=2$ (four blocks), is uncharacterized: is there a clean Diophantine condition analogous to Theorem~\ref{thm:csdiscriminants}, or — as the block count grows — does the governing polynomial's degree grow with $m$ in a way that forecloses a uniform closed-form answer?
	\end{openproblem}
	
	\section{Conclusion}
	\label{sec:conclusion}
	
	Threshold graphs occupy a genuinely unusual position in spectral graph theory: a single combinatorial object — the binary creation sequence — drives exact or near-exact spectral computations across essentially every matrix the field has studied, from the classical adjacency and Laplacian matrices through the modern distance-based and eccentricity matrices. The common machinery throughout is the equitable partition induced by twin vertices, reducing an $n$-dimensional spectral problem to a quotient matrix of size proportional to the number of alternations in the creation sequence — together with, for several matrices, a master identity connecting the matrix of interest to a better-understood relative (most notably the distance-Laplacian-to-Laplacian identity of Proposition~\ref{prop:master}).
	
	The literature has matured furthest for the adjacency matrix (essentially complete: eigenvalue location, cospectrality, energy, and eigenvalue-free intervals are all settled) and the ordinary and distance Laplacians (both admit exact closed forms via Merris's theorem and its master-identity consequence). It is comparatively less developed for the distance signless Laplacian (Open Problem~\ref{op:dqthreshold}) and for several natural generalizations — quasi-threshold graphs, weighted threshold graphs, and threshold hypergraphs — each the subject of active work as recently as 2024 and 2025. The cross-cutting integrality study of Section~\ref{sec:integral} reinforces this same divide from a new angle: $L$ and $D^L$ are unconditionally integral for every threshold graph via the master identity, while $A$, $S$, $D$, $D^Q$, and $\varepsilon$ each require a genuine Diophantine condition, fully resolved here only for the simplest ($m=1$) family. We expect the open problems collected in Section~\ref{sec:open}, together with the still-developing extremal spectral radius literature of Section~\ref{sec:extremal} and the integrality questions of Section~\ref{sec:integral}, to remain fruitful directions for the next several years of research on this exceptionally well-behaved graph class.
	
	\medskip
	We conclude this study with two tables that show the eigenvalue-free intervals of threshold graphs and chain graphs based on multiple graph matrices, which do not include any non-trivial eigenvalues.
	\newpage
	
	\begin{table}[h]
		\centering
		\begin{tabular}{lll}
			\toprule
			Sl. No. & Matrix & Eigenvalue-free Interval \\
			\midrule
			1. & $A$ & $(-1,0)$  \cite{JacobsTrevisanTura2013} \\
			2. & $A$ & $\left[\frac{-1-\sqrt2}{2},\ \frac{-1+\sqrt2}{2}\right] \cite{Ghorbani2019}  $ \\
			3. & $D$ & $(-2,-1)$ \cite{LuHuangLou2018}\\
			4. & $D$ & $\left[\frac{-3-\sqrt2}{2},\ -1\right)$ \cite{AlazemiEtAl2021} \\
			5. & $ \varepsilon$ &$(-2,-1)$ \cite{QiuTang2022}\\
			6. & $\varepsilon$ & $(-1-\sqrt2,-2)\cup (-2,-1)$ \cite{AndelicEccentricity2022}\\
			7. & $S$&$(-\sqrt2,\sqrt2)$ \cite{XiongHou2022}\\
			8.&$L$&Not addressed\\
			9.&$D^L$&Not addressed\\
			10.&$D^Q$&Not addressed\\
			\bottomrule
		\end{tabular}
		\caption{Eigenvalue-free interval of threshold graphs.}
		\label{tab:thresholdeigenvaluefreeinterval}
	\end{table}

	\begin{table}[h]
		\centering
		\begin{tabular}{lll}
			\toprule
			Sl. No. & Matrix & Eigenvalue-free Interval \\
			\midrule
			1. & $A$ & $[-\frac{1}{2},0)\cup (0,\frac{1}{2}]$  \cite{AlazemiEtAl2016} \\
			2. & $D$ & $(-3,-2)$ \cite{AlazemiEtAl2021}\\
			3. & $S$ & $[-2,0]$ \cite{XiongHou2024}\\
			4. & $\varepsilon,L,D^L,D^Q$ & Not addressed \\
			
			\bottomrule
		\end{tabular}
		\caption{Eigenvalue-free interval of chain graphs.}
		\label{tab:chaineigenvaluefreeinterval}
	\end{table}

	
	\section*{Acknowledgments}
	
	The research of Pallabi Manna is supported by Anusandhan National Research Foundation
	(ANRF)-India under the National Postdoctoral Fellowship scheme (File No.\ PDF/2025/006751). This
	author also acknowledges the excellent research facilities provided by the Department of Mathematics,
	IIT Hyderabad, and the mentorship of Dr.\ Rajesh Kannan for his constant support and motivation. The second author acknowledges excellent research facilities provided by SRM University-AP.
	
	\section*{Statements and Declarations}
	
	\noindent\textbf{Competing Interests.} The authors have no competing interests to declare that are
	relevant to the content of this article.
	
	\section*{Data Availability}
	
	Data sharing is not applicable to this article, as no datasets were used or generated during this
	study.

\end{document}